\documentclass{amsart}

\usepackage{amsmath, enumerate, amsfonts, amssymb, amsthm, graphics, graphicx, verbatim, caption, subcaption, MnSymbol, booktabs, multirow, soul, array, hyperref, float, comment}
\usepackage[numbers]{natbib}
\usepackage[table]{xcolor}
\usepackage[english]{babel}
\usepackage{quiver}
\usepackage[margin=0.95in]{geometry}
\usepackage{tikz, float}
\usetikzlibrary{shapes}
\usetikzlibrary{positioning}
\usetikzlibrary{decorations.markings}
\usetikzlibrary{arrows}
\tikzstyle{subgroup}=[scale=1]

\newtheorem{prevtheorem}{Theorem}
\newtheorem{theorem}{Theorem}
\newtheorem{proposition}[theorem]{Proposition}
\newtheorem{corollary}[theorem]{Corollary}

\def\CD#1{\mathcal{CD}#1}

\title{The Chermak--Delgado Measure as a Map on Posets}

\author{William Cocke}
\address{Language Technology Institute, Carnegie Mellon University, Pittsburgh, PA 15213}
\email{cocke@cmu.edu}

\author{Ryan McCulloch}
\address{Department of Mathematics, University of Bridgeport, Bridgeport, CT 06604}
\email{rmccullo1985@gmail.com}

\date{\today}

\begin{document}

\begin{abstract}
The Chermak--Delgado measure of a finite group is a function which assigns to each subgroup a positive integer. In this paper, we give necessary and sufficient conditions for when the Chermak--Delgado measure of a group is actually a map of posets, i.e., a monotone function from the subgroup lattice to the positive integers. We also investigate when the Chermak--Delgado measure, restricted to the centralizers, is increasing. 
\end{abstract}



\maketitle

\section{Introduction.}
The Chermak--Delgado measure is a function from the subgroup lattice of a finite group $G$ to the positive integers that maps a subgroup $H$ to $|H|\cdot |\mathbf{C}_G(H)|$. We write the Chermak--Delgado measure for a finite group $G$ as $m_G$. Chermak and Delgado \cite{Che_89} introduced the function $m_G$ in a slightly more general form. The function $m_G$ has the peculiar property that the subgroups $H\leq G$ such that $m_G(H)$ has maximal value form a sublattice, the so-called Chermak--Delgado lattice of $G$. We write $\mathcal{CD}(G)$ for the Chermak--Delgado lattice of $G$. A proof that the fiber of the maximum value of $m_G$ forms a lattice and other basic properties of $\mathcal{CD}(G)$ can be found in Isaacs \cite[Section 1.E]{FGT}. More modern research on the Chermak--Delgado often focuses on finding groups such that the Chermak--Delgado lattice has a specified shape \cite{An_22_2}, or on specific groups \cite{Mor_18}, or other structural properties of the lattice itself \cite{Bre_14_2, Bur_23, Mcc_18}. We direct students interested in the Chermak--Delgado lattice  to the article by Wilcox \cite{Wil_16}, in which she gives a non-exhaustive list of potential research directions suitable for students. 

We write $\mathcal{S}(G)$ for the the subgroup lattice of $G$. Then $m_G$ is a map from $\mathcal{S}(G)$ to the positive integers $\mathbb{Z}^{+}$. In Figure \ref{fig: S3} we represent the map $m_G$ where $G = S_3$, the symmetric group on 3 symbols. 

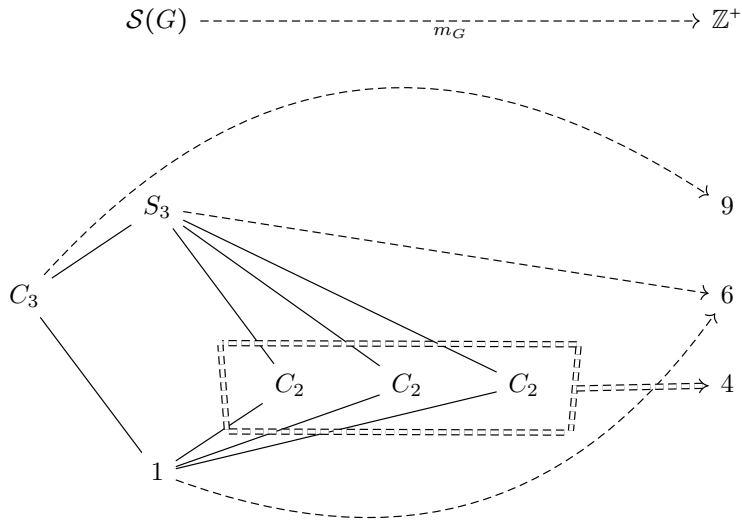
\begin{figure}[H]
    \centering
\begin{tikzcd}
	& {\mathcal{S}(G)} &&&&& {\mathbb{Z}^+} \\
	\\
	\\
	& {S_3} &&&&& 9 \\
	{C_3} && {} && {} && 6 \\
	& {} & {C_2} & {C_2} & {C_2} & {} & 4 \\
	& 1 & {} && {}
	\arrow[no head, from=7-2, to=5-1]
	\arrow[no head, from=5-1, to=4-2]
	\arrow[no head, from=6-5, to=4-2]
	\arrow[no head, from=6-3, to=4-2]
	\arrow[no head, from=6-4, to=4-2]
	\arrow[no head, from=7-2, to=6-3]
	\arrow[no head, from=7-2, to=6-4]
	\arrow[no head, from=7-2, to=6-5]
	\arrow[curve={height=60pt}, dashed, from=7-2, to=5-7]
	\arrow[dashed, from=4-2, to=5-7]
	\arrow[curve={height=-80pt}, dashed, from=5-1, to=4-7]
	\arrow["{m_G}"', dashed, from=1-2, to=1-7]
	\arrow[""{name=0, anchor=center, inner sep=0}, draw=none, from=5-3, to=6-2]
	\arrow[""{name=1, anchor=center, inner sep=0}, draw=none, from=6-2, to=7-3]
	\arrow[""{name=2, anchor=center, inner sep=0}, draw=none, from=6-6, to=7-5]
	\arrow[""{name=3, anchor=center, inner sep=0}, draw=none, from=5-5, to=6-6]
	\arrow[Rightarrow, dashed, no head, from=1, to=2]
	\arrow[""{name=4, anchor=center, inner sep=0}, Rightarrow, dashed, no head, from=3, to=2]
	\arrow[Rightarrow, dashed, no head, from=3, to=0]
	\arrow[Rightarrow, dashed, no head, from=1, to=0]
	\arrow[Rightarrow, dashed, from=4, to=6-7]
\end{tikzcd}
    \caption{A diagram for $G=S_3$ showing the Chermak--Delgado measure $m_G:\mathcal{S}(G)\rightarrow \mathbb{Z}^+$. All of the subgroups isomorphic to $C_2$ have measure $4$.}
    \label{fig: S3}
\end{figure}

As suggested by the diagram in Figure \ref{fig: S3}, both the domain and the codomain of the map $m_G$ are posets; where we take $x\preceq y$ in $\mathbb{Z}^{+}$ to mean that $x\leq y$. However, the function $m_G$ does not in general preserve this poset structure. To use the common vernacular of category theory, we could say that $m_G$ forgets the poset structure of its domain and codomain. In this paper we investigate 
when the map $m_G$ is actually a map of posets, i.e., when $m_G:\mathcal{S}(G)\rightarrow \mathbb{Z}^{+}$ is monotone increasing or decreasing. The question of when $m_G$ is a morphism of posets was investigated by Fasol{\u{a}} \cite{Fas_24} and we explain this connection in Section \ref{sec: back}. 

We now present two examples of finite groups together with their Chermak--Delgado measures, one such that the measure is monotone, and one where the measure is not monotone. In Figure \ref{fig: D8} we present the subgroup lattice and Chermak--Delgado measure of the dihedral group of order 8 (written as $D_8$). The subgroup lattice can be partitioned into two parts, each containing $5$ subgroups. Each of the nontrivial normal subgroups has a Chermak--Delgado measure of $16$, while each of the other five subgroups has a Chermak--Delgado measure of 8.  Hence the Chermak--Delgado measure from $D_8$ to $\{8,16\}$ forms a poset. In contrast, the Chermak--Delgado measure from $A_4$, the alternating group on 4 symbols, is not monotone. In Figure \ref{fig: A4} we present a diagram showing the map $m_G$ for $G=A_4$.

\begin{figure}
    \centering
\begin{tikzcd}
	&&& {\mathcal{S}(G)} &&& {\mathbb{Z}^{+}} \\
	&& {} \\
	&&& {D_8} & {} && 16 \\
	{} & {} & {K_4} & {C_4} & {K_4} & {} & {} \\
	& {C_2} & {C_2} & {C_2} & {C_2} & {C_2} && {} \\
	{} &&& 1 &&& 8 \\
	&& {} && {}
	\arrow[no head, from=4-3, to=3-4]
	\arrow[no head, from=4-5, to=3-4]
	\arrow[no head, from=5-4, to=4-3]
	\arrow[no head, from=5-4, to=4-5]
	\arrow["{m_G}", dashed, from=1-4, to=1-7]
	\arrow[no head, from=5-2, to=6-4]
	\arrow[no head, from=5-3, to=6-4]
	\arrow[""{name=0, anchor=center, inner sep=0}, no head, from=5-4, to=6-4]
	\arrow[no head, from=5-5, to=6-4]
	\arrow[no head, from=5-6, to=6-4]
	\arrow[no head, from=5-2, to=4-3]
	\arrow[no head, from=5-3, to=4-3]
	\arrow[no head, from=5-4, to=4-4]
	\arrow[no head, from=5-5, to=4-5]
	\arrow[no head, from=5-6, to=4-5]
	\arrow[""{name=1, anchor=center, inner sep=0}, draw=none, from=4-2, to=4-3]
	\arrow[""{name=2, anchor=center, inner sep=0}, draw=none, from=4-5, to=4-6]
	\arrow[draw=none, from=4-7, to=6-7]
	\arrow[draw=none, from=4-7, to=6-7]
	\arrow[""{name=3, anchor=center, inner sep=0}, draw=none, from=5-6, to=5-8]
	\arrow[""{name=4, anchor=center, inner sep=0}, draw=none, from=7-3, to=7-5]
	\arrow[""{name=5, anchor=center, inner sep=0}, draw=none, from=4-1, to=6-1]
	\arrow[""{name=6, anchor=center, inner sep=0}, draw=none, from=2-3, to=3-5]
	\arrow[no head, from=4-4, to=3-4]
	\arrow[dashed, Rightarrow, no head, from=1, to=0]
	\arrow[dashed, Rightarrow, no head, from=2, to=0]
	\arrow[dashed, no head, from=3, to=2]
	\arrow[""{name=7, anchor=center, inner sep=0}, dashed, no head, from=4, to=3]
	\arrow[dashed, no head, from=5, to=4]
	\arrow[dashed, no head, from=1, to=5]
	\arrow[dashed, Rightarrow, no head, from=1, to=6]
	\arrow[""{name=8, anchor=center, inner sep=0}, dashed, Rightarrow, no head,  from=6, to=2]
	\arrow[shorten <=8pt, Rightarrow, dashed, from=7, to=6-7]
	\arrow[shorten <=6pt, Rightarrow, dashed, from=8, to=3-7]
\end{tikzcd}
    \caption{A diagram for $G=D_8$ showing the Chermak--Delgado measure $m_G:\mathcal{S}(G)\rightarrow \mathbb{Z}^+$. The two regions of the diagram are the fibers of $m_G$. The map $m_G$ is a monotone map of posets as well. } \label{fig: D8}
\end{figure}
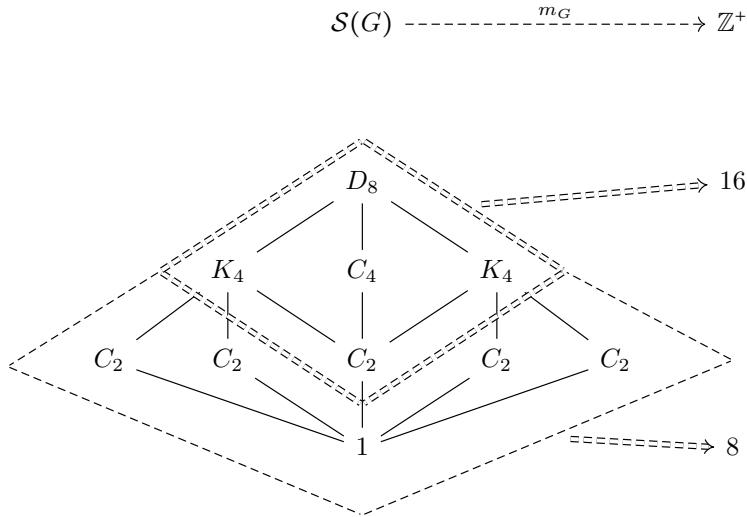

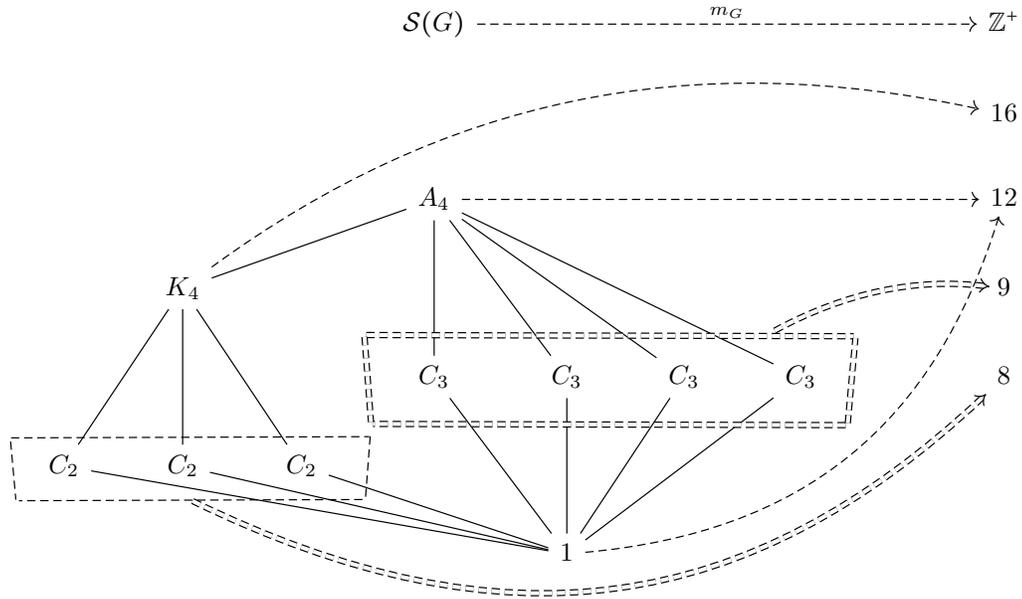
\begin{figure}
    \centering
\begin{tikzcd}
	&&&& {\mathcal{S}(G)} &&&&& {\mathbb{Z}^{+}} \\
	&&&&&&&&& 16 \\
	&&&& {A_4} &&&&& 12 \\
	&& {K_4} & {} & {} && {} & {} && 9 \\
	& {} && {} & {C_3} & {C_3} & {C_3} & {C_3} & {} & 8 \\
	{} & {C_2} & {C_2} & {C_2} & {} &&& {} \\
	& {} && {} && 1
	\arrow[no head, from=4-3, to=3-5]
	\arrow[no head, from=5-6, to=3-5]
	\arrow[no head, from=7-6, to=5-6]
	\arrow[dashed, from=3-5, to=3-10]
	\arrow[curve={height=50pt}, dashed, from=7-6, to=3-10]
	\arrow["{m_G}", dashed, from=1-5, to=1-10]
	\arrow[curve={height=-50pt}, dashed, from=4-3, to=2-10]
	\arrow[no head, from=4-3, to=6-2]
	\arrow[no head, from=4-3, to=6-3]
	\arrow[no head, from=4-3, to=6-4]
	\arrow[no head, from=6-2, to=7-6]
	\arrow[no head, from=6-3, to=7-6]
	\arrow[no head, from=6-4, to=7-6]
	\arrow[no head, from=5-5, to=7-6]
	\arrow[no head, from=5-7, to=7-6]
	\arrow[no head, from=5-8, to=7-6]
	\arrow[no head, from=3-5, to=5-5]
	\arrow[no head, from=5-7, to=3-5]
	\arrow[no head, from=5-8, to=3-5]
	\arrow[""{name=0, anchor=center, inner sep=0}, draw=none, from=5-2, to=6-1]
	\arrow[""{name=1, anchor=center, inner sep=0}, draw=none, from=6-1, to=7-2]
	\arrow[""{name=2, anchor=center, inner sep=0}, draw=none, from=5-4, to=6-5]
	\arrow[draw=none, from=4-5, to=4-4]
	\arrow[""{name=3, anchor=center, inner sep=0}, draw=none, from=4-5, to=5-4]
	\arrow[""{name=4, anchor=center, inner sep=0}, draw=none, from=5-9, to=4-8]
	\arrow[""{name=5, anchor=center, inner sep=0}, draw=none, from=5-9, to=6-8]
	\arrow[""{name=6, anchor=center, inner sep=0}, draw=none, from=6-5, to=7-4]
	\arrow[""{name=7, anchor=center, inner sep=0}, draw=none, from=4-7, to=5-7]
	\arrow[shorten <=5.5pt, shorten >=4.5pt, dashed, no head, from=0, to=1]
	\arrow[shift right=2, dashed, no head, from=0, to=2]
	\arrow[Rightarrow, dashed, no head, from=3, to=4]
	\arrow[Rightarrow, dashed, no head, from=3, to=2]
	\arrow[Rightarrow, dashed, no head, from=5, to=2]
	\arrow[""{name=8, anchor=center, inner sep=0}, shift right=2, dashed, no head, from=6, to=1]
	\arrow[shorten <=4pt, shorten >=3.5pt, dashed, no head, from=2, to=6]
	\arrow[Rightarrow, dashed, no head, from=5, to=4]
	\arrow[""{name=9, anchor=center, inner sep=0}, draw=none, from=7, to=4]
	\arrow[curve={height=75pt}, Rightarrow, dashed, from=8, to=5-10]
	\arrow[curve={height=-10pt}, shorten <=1.3pt, shorten >= -1.5pt, Rightarrow, dashed, from=9, to=4-10]
\end{tikzcd}
    \caption{A diagram for $G=A_4$ showing the Chermak--Delgado measure $m_G:\mathcal{S}(G)\rightarrow \mathbb{Z}^+$. We note that each copy of $C_3$ has measure $9$ and all of the copies of $C_2$ have measure 8. Since $1$ has measure 12, the map $m_G$ is not a monotone map of posets.} \label{fig: A4}
\end{figure}

Recall that there are two forms of monotone maps, increasing and decreasing. We show in Proposition \ref{prop: decreasing} that if $m_G$ is a decreasing function from $\mathcal{S}(G)$ to $G$, then $G=1$. Consequently, we restrict our attention to the case in which the Chermak--Delgado measure is increasing for the rest of the paper. 

Moreover, if $m_G$ is increasing, then the Chermak--Delgado lattice is exactly the interval of subgroups between $\mathbf{Z}(G)$ and $G$, i.e., $H\in \mathcal{CD}(G)$ if and only if $\mathbf{Z}(G) \leq H$ (and of course $H\leq G)$. We use the notation $[G/\mathbf{Z}(G)]$ for the interval of $\mathcal{S}(G)$ between $G$ and $\mathbf{Z}(G)$. The following theorem shows that the necessary conditions discovered by Fasol{\u{a}} \cite{Fas_24} are in fact sufficient. 

\begin{prevtheorem} \label{prev: increasing}
    Let $G$ be a finite group. Then the following are equivalent.
    \begin{enumerate}
        \item The function $m_G$ is increasing on $\mathcal{S}(G)$. 
        \item $m_G(H) = m_G(H\cap \mathbf{Z}(G))$ for all subgroups $H$ of $G$. 
        \item $\mathcal{CD}(G) = [G/\mathbf{Z}(G)]$.
    \end{enumerate} 
\end{prevtheorem}

Of note, since the Chermak--Delgado lattice is self-dual under taking centralizers, if $\mathcal{CD}(G) = [G/\mathbf{Z}(G)]$, then all of the subgroups containing $\mathbf{Z}(G)$ are centralizers. We write $\mathcal{C}(G)$ for the set of subgroups of $G$ that are centralizers in $G$ (which is a lattice). Theorem \ref{prev: increasing} implies that if $m_G$ is increasing then $\CD(G) = \mathcal{C}(G)$.

Moreover, if we relax this condition to considering $m_G$ as a function not from $\mathcal{S}(G)$ but from $\mathcal{C}(G)$, we obtain a similar result. As an example of a sort of linguistic duality, we cannot help but point out that \emph{relaxing} the condition on $m_G$ to only be increasing on $\mathcal{C}(G)$ is equivalent to \emph{restricting} the domain of $m_G$. 

\begin{prevtheorem}\label{prev: centralizer_increasing}
Suppose $G$ is a finite group.  Then $m_G$ is increasing on $\mathcal{C}(G)$ if and only if $\mathcal{C}(G) = \CD(G)$, and in such case the group $G$ is nilpotent.
\end{prevtheorem}

We note that since all centralizers of $G$ contain $\mathbf{Z}(G)$, the identification of $\CD(G)$ in Theorems \ref{prev: increasing} and \ref{prev: centralizer_increasing} are both specific instances of the following.
\begin{prevtheorem}\label{prev: main}
    Let $G$ be a finite group. If the Chermak--Delgado measure $m_G$ is increasing on a poset $\mathcal{P}$ of subgroups of $G$ such that \begin{enumerate}
        \item $\mathbf{Z}(G) \in \mathcal{P}$;
        \item $G\in \mathcal{P}$;
        \item $\mathcal{P} \cap \CD(G) \neq \varnothing$;
   \end{enumerate}
    then $\CD(G) \cap \mathcal{P} = \{H \in \mathcal{P} : \mathbf{Z}(G) \leq H.\}$ 
\end{prevtheorem}

The paper proceeds as follows. In Section \ref{sec: back_mat} we present background material, including a discussion of Fasol{\u{a}}'s prior results, and general results we will use. Then in Section \ref{sec: proofs} we prove Theorem \ref{prev: increasing} and present a classification of finite groups with $m_G$ monotone on the full subgroup lattice. In Section \ref{sec: centralizers} we prove Theorem \ref{prev: centralizer_increasing} and investigate when $m_G$ is monotone on the centralizer lattice.

\section{Background.} \label{sec: back_mat}
In this section we gather some relevant facts about centralizers, the Chermak--Delgado lattice, and subnormal subgroups. We also discuss prior work on when the Chermak--Delgado measure is increasing (see Section \ref{sec: back}). 

\begin{proposition}\label{prop: basic_cent}
Recall the following facts about centralizers. Let $C \in \mathcal{C}(G)$, i.e., there exists a subgroup $H \leq G$ such that $C=\mathbf{C}_G(H)$. 
\begin{enumerate}
    \item The subgroup $C$ is the centralizer of its centralizer, i.e.,  $C = \mathbf{C}_G(\mathbf{C}_G(C))$.
    \item The subgroup $C$ and its centralizer have the same Chermak--Delgado measure, i.e., $m_G(C) = |C| \cdot |\mathbf{C}_G(C)| = |\mathbf{C}_G(C)| \cdot |\mathbf{C}_G(\mathbf{C}_G(C))| = m_G(\mathbf{C}_G(C))$.
\end{enumerate}
\end{proposition}

Note that for any subset $S$ of a group $G$, $\mathbf{C}_G(S) = \mathbf{C}_G(\langle S \rangle) \in \mathcal{C}(G)$, and so, in particular, for any element $x \in G$, the subgroup $\mathbf{C}_G(x) \in \mathcal{C}(G)$. 

\begin{proposition}[{\cite[Theorem 1.44]{FGT}}]
Let $G$ be a finite group.  Then $\CD(G) \subseteq \mathcal{C}(G)$.
\end{proposition}

The next proposition will be used throughout the paper.  

\begin{proposition}
    Suppose $m_G$ is increasing on a subset $\mathcal{P}$ of $\mathcal{S}(G)$.  If an abelian centralizer $C \in \mathcal{P}$ and $\mathbf{C}_G(C) \in \mathcal{P}$, then $m_G$ is constant on $\{ X \in \mathcal{P} \,\, : \,\, C \leq X \leq \mathbf{C}_G(C) \} = [\mathbf{C}_G(C)/C]_{\mathcal{P}}$.
\end{proposition}

\begin{proof}
Let $X\in \mathcal{P}$ such that $C \leq X \leq \mathbf{C}_G(C)$.  Since $m_G$ is increasing, we have $m_G(C) \leq m_G(X) \leq m_G(\mathbf{C}_G(C))$ and thus $m_G(X) = m_G(C)$. 
\end{proof}

\begin{corollary}\label{cor: gen_inc}
    Suppose $m_G$ is increasing on a subset $\mathcal{P}$ of $\mathcal{S}(G)$.  If an abelian centralizer $C \in \mathcal{P}$ and $\mathbf{C}_G(C) \in \mathcal{P}$, and if $[\mathbf{C}_G(C)/C]_{\mathcal{P}} \cap \CD(G) \neq \emptyset$, then $[\mathbf{C}_G(C)/C]_{\mathcal{P}} \subseteq \CD(G)$.
\end{corollary}

Let $\mathcal{S}\textit{n}(G)$ denote the collection of all subnormal subgroups of a group $G$.  Wielandt showed that for a finite group $G$, $\mathcal{S}\textit{n}(G)$ is a sublattice of $\mathcal{S}(G)$.

\begin{proposition}[Wielandt {\cite[Lemma 2.5]{FGT}}]\label{prop: sub-join}
 Let $G$ be a finite group and suppose that $S,T$ are subnormal subgroups of $G$.  Then $\langle S, T \rangle$ is subnormal in $G$.
\end{proposition}

Brewster \& Wilcox in \cite{Bre_12} showed that subgroups in $\mathcal{CD}(G)$ are subnormal, see also \cite{Coc_20}.

\begin{proposition}\label{prop: CD_Sn}
Let $G$ be a finite group.  Then $\CD(G) \subseteq \mathcal{S}\textit{n}(G)$.
\end{proposition}

\subsection{Previous work on when \texorpdfstring{$m_G$}{mG} is increasing} \label{sec: back}
This work was directly inspired by the 2024 paper of Georgiana Fasol{\u{a}} \cite{Fas_24}.  
 Fasol{\u{a}} examined the question of when subgroups with minimal Chermak--Delgado measure form a sublattice. Her partial answers to this question relied on two additional  hypotheses: 
\begin{enumerate}
    \item A bound on the number of values taken by the Chermak--Delgado measure $m_G$.
    \item Assuming that the Chermak--Delgado measure $m_G$ is increasing, i.e., $H\leq K \Rightarrow m_G(H) \leq m_G(K)$. \label{list: increasing}
\end{enumerate}

Of note, Fasol{\u{a}} \cite{Fas_24} showed that $m_G$ is increasing implies that $G$ is nilpotent. She also proved that $(1)$ implies $(2)$ and $(3)$ in Theorem \ref{prev: increasing}. 

\section{Monotone Chermak--Delgado measure on all subgroups.} \label{sec: proofs}

In this section, we prove Theorem \ref{prev: increasing}, which states that the function $m_G$ is increasing on the subgroup lattice if and only if $\mathcal{CD}(G) = [G/\mathbf{Z}(G)]$. After proving Theorem \ref{prev: increasing}, we then collect results showing when this condition is satisfied. 

We start by showing that  the Chermak--Delgado measure is decreasing only on the trivial group. 

\begin{proposition}\label{prop: decreasing}
    Let $G$ be a finite group. If the Chermak--Delgado measure $m_G$ on $G$ satisfies $K \leq H \Rightarrow m_G(H) \leq m_G(K)$, then $G=1$. 
\end{proposition}
\begin{proof}
    We note that $m_G(G) \leq m_G(1)$ implies that $\mathbf{Z}(G) =1$ and $m_G(G) = m_G(1)$.  Now for all $H\leq G$ we have  $m_G(G) \leq m_G(H)\leq m_G(1)$, and thus all elements of $\mathcal{S}(G)$ have the same Chermak--Delgado measure. The only group $G$ such that the image of $m_G$ has size $1$ is $1$. To see this, if $S$ is a nontrivial Sylow $p$-subgroup of $G$, we have that $m_G(S)\neq m_G(1)$, since $|\mathbf{C}_G(S)|$ is divisible by $p$. Therefore, the Sylow subgroups of $G$ are trivial and we conclude that $G=1$.   
\end{proof}

We now prove Theorem \ref{prev: increasing}. Since the theorem involved a few hypothesis, we restate it before giving the proof. 

\setcounter{prevtheorem}{0}
\begin{prevtheorem}
    Let $G$ be a finite group. Then the following are equivalent.
    \begin{enumerate}
        \item The function $m_G$ is increasing on $\mathcal{S}(G)$. 
        \item $m_G(H) = m_G(H\cap \mathbf{Z}(G))$ for all subgroups $H$ of $G$. 
        \item $\mathcal{CD}(G) = [G/\mathbf{Z}(G)]$.
    \end{enumerate}
\end{prevtheorem}

\begin{proof}
    $(1) \rightarrow (2)$ is Lemma 2.3 in \cite{Fas_24}.  Note that if $m_G(H) = m_G(H\cap \mathbf{Z}(G))$ for all subgroups $H$ of $G$, then $m_G(H) = m_G(\mathbf{Z}(G))$ for all subgroups $H$ that contain $\mathbf{Z}(G)$. This implies that $m_G$ is constant on $[G/\mathbf{Z}(G)]$, and hence $\mathcal{CD}(G) = [G/\mathbf{Z}(G)]$.  Thus, $(2) \rightarrow (3)$.

    Now we show that $(3) \rightarrow (1)$.
    Suppose that $\mathcal{C}\mathcal{D}(G) = [G/\mathbf{Z}(G)]$. Let $H \leq K$ be two subgroups of $G$. We write $Z$ for $\mathbf{Z}(G)$. Then $\mathbf{C}_G(HZ) = \mathbf{C}_G(H)$ and $\mathbf{C}_G(KZ) = \mathbf{C}_G(K)$.  Since both $HZ$ and $ KZ$ are in $ \mathcal{C}\mathcal{D}(G)$, we know that $m_G(HZ) = |HZ||\mathbf{C}_G(H)| = |KZ||\mathbf{C}_G(K)| = m_G(KZ)$.  And so $$\frac{|H||Z|}{|H \cap Z|} |\mathbf{C}_G(H)| = \frac{|K||Z|}{|K \cap Z|} |\mathbf{C}_G(K)|.$$ Reordering terms, this is 
 \begin{equation}\tag{i}\label{eqn: inc 1} m_G(H) \cdot \frac{|Z|}{|H\cap Z|} = m_G(K) \cdot \frac{|Z|}{|K\cap Z|}.
 \end{equation}

 Since $H \leq K$, we have $H \cap Z \leq K \cap Z$. Thus  \begin{equation}\tag{ii}\label{eqn: inc 2} \frac{|Z|}{|H \cap Z|} \geq \frac{|Z|}{|K \cap Z|}.\end{equation} 

Combining equations \ref{eqn: inc 1} and \ref{eqn: inc 2} shows that $m_G(H) \leq m_G(K)$. 
\end{proof}

\begin{proposition}
    Let $G$ be a finite group, let $k$ be the number of divisors of $|\mathbf{Z}(G)|$, and consider $m_G$ as a map from $\mathcal{S}(G)$ to $\mathbb{Z}^+$.  Then $|Im(m_G)| \geq k$.
\end{proposition}

\begin{proof}
Since $\mathbf{Z}(G)$ is abelian, $\mathbf{Z}(G)$ possesses a subgroup $H$ of order $n$ for each $n$ dividing $|\mathbf{Z}(G)|$, and $m_G(H) = n \cdot |G|$.  Thus $\{ n\cdot|G| \,\, : \,\, n \,\, \mid \,\, |\mathbf{Z}(G)| \} \subseteq Im(m_G)$, and the result follows.
\end{proof}

The class of groups $G$ in Theorem \ref{prev: increasing} achieve this lower bound on the size of the image of $m_G$.

\begin{corollary}
    Let $G$ be a finite group, and let $k$ be the number of divisors of $|\mathbf{Z}(G)|$.  If $m_G$ is increasing on $\mathcal{S}(G)$, then $|Im(m_G)| = k$.
\end{corollary}

\begin{proof}
    If $m_G$ is increasing on $\mathcal{S}(G)$, then $m_G(H) = m_G(H\cap \mathbf{Z}(G))$ for all subgroups $H$ of $G$ ($(1) \rightarrow (2)$ in Theorem \ref{prev: increasing}), and it follows that $\{ n\cdot|G| \,\, : \,\, n \,\, \mid \,\, |\mathbf{Z}(G)| \} = Im(m_G)$.
\end{proof}

An open question is ``What can be said about groups $G$ with $|Im(m_G)| = k$, where $k$ be the number of divisors of $|\mathbf{Z}(G)|$?''.  More questions along these lines can be found in \cite{Tar_20}.

A classification of groups $G$ having $\mathcal{CD}(G) = [G/\mathbf{Z}(G)]$ appears in \cite{Tar_18}, and is based on the classification of groups $G$ having $\mathcal{C}(G) = [G/\mathbf{Z}(G)]$ that appears in \cite{Sch94}.  We collect two results of Cheng, both in \cite{Cheng_82}, that provide the classification for the $p$-group case.  Corollary \ref{cor: nilp_classif} in the next section yields a classification of the general case that is equivalent to the one in \cite{Tar_18}.

\begin{proposition} \label{prop: Cheng}
    Let $G$ be a finite $p$-group.  Then $\mathcal{C}(G) = [G/\mathbf{Z}(G)]$ if and only if $G' = \langle a \rangle$ is cyclic and $[a,G] \leq \langle a^4 \rangle$.  If $p>2$, then, clearly $[a,G] \leq \langle a^4 \rangle$, and so in that case $\mathcal{C}(G) = [G/\mathbf{Z}(G)]$ if and only if $G'$ is cyclic.
\end{proposition}

Note that the case $G$ abelian is included in Proposition \ref{prop: Cheng}.

\begin{proposition}
    If $G$ is as in Proposition \ref{prop: Cheng}, $Z = \mathbf{Z}(G)$, and $Z \leq H \leq G$, then $$|G/Z| = |H/Z| \cdot |\mathbf{C}_G(H)/Z|.$$
\end{proposition}

It is clear that the above two propositions together imply that a $p$-group $G$ has $\CD(G) = \mathcal{C}(G) = [G/\mathbf{Z}(G)]$ if and only if $G$ is as in Proposition \ref{prop: Cheng}.  This was also noted in \cite{Tar_18}.

\section{Monotone Chermak--Delgado Measure on Centralizers.}\label{sec: centralizers}

In this section, we prove Theorem \ref{prev: centralizer_increasing} which classifies when the Chermak--Delgado measure is increasing on centralizer subgroups of $G$.  

\begin{proposition}\label{prop: sub-nilp}
 Let $G$ be a finite group.  Then the following are equivalent.  
 \begin{enumerate}
 \item $G$ is nilpotent. 
 \item $\mathcal{S}(G) = \mathcal{S}\textit{n}(G)$.
 \item $\mathcal{C}(G) \subseteq \mathcal{S}\textit{n}(G)$.
 \end{enumerate}
\end{proposition}

\begin{proof}
$(1) \leftrightarrow (2)$ is Lemma 2.1 in \cite{FGT}.

$(2) \rightarrow (3)$ is a tautological weakening of hypotheses. 

To see that $(3) \rightarrow (2)$, let $H$ be an arbitrary subgroup of $G$.  Now for any $x \in G$ we have $\langle x \rangle \unlhd \mathbf{C}_G(x)$ which is subnormal in $G$.  Hence $H = \langle \langle x \rangle \,\, | \,\, x \in H \rangle$ is subnormal by Proposition \ref{prop: sub-join}.
\end{proof}

We now prove Theorem \ref{prev: centralizer_increasing} stating that the Chermak--Delgado measure is increasing on centralizers if and only if the Chermak--Delgado lattice is exactly the centralizer lattice. Moreover if the Chermak--Delgado measure is increasing on centralizers, then the group is nilpotent. 

\begin{proof}[Proof of Theorem \ref{prev: centralizer_increasing}] 
We can apply Corollary \ref{cor: gen_inc} taking $\mathcal{P}$ to be $\mathcal{C}(G)$. This gives us that $\mathcal{C}(G) \subseteq \CD(G)$. Since $\CD(G)$ consists of centralizers we conclude that $\CD(G) = \mathcal{C}(G)$. 

By Proposition \ref{prop: CD_Sn}, $\CD(G) \subseteq \mathcal{S}\textit{n}(G)$, and so the nilpotency of $G$ follows from Proposition \ref{prop: sub-nilp}.
\end{proof}

Recall that finite nilpotent groups are direct products of their Sylow subgroups. Moreover, both the Chermak--Delgado lattice and the centralizer lattice respect direct products.  For the Chermak--Delgado lattice, the result appears in \cite{Bre_12}, and for the centralizer lattice it is an exercise in \cite{Sch94}. We collect this together in the proposition below. 

\begin{proposition}
    Let $G_1,\dots,G_n$ be finite groups.  Then 
\begin{itemize}
    \item $\mathcal{C}(G_1 \times \cdots \times G_n) = \mathcal{C}(G_1) \times \cdots \times \mathcal{C}(G_n)$.
    \item $\CD(G_1 \times \cdots \times G_n) = \CD(G_1) \times \cdots \times \CD(G_n)$.
\end{itemize}
\end{proposition}

We immediately obtain the following. 

\begin{corollary}\label{cor: inc_decom}
  Suppose $G$ is a finite group.  Then $m_G$ is increasing on $\mathcal{C}(G)$ if and only if $G = P_1 \times \cdots \times P_n$, with $p_1, \dots, p_n$ distinct primes and each $P_i$ a $p_i$-group with $\mathcal{C}(P_i) = \CD(P_i)$.
\end{corollary}

Restricting to the case of Theorem \ref{prev: increasing} we obtain a classification.

\begin{corollary}\label{cor: nilp_classif}
  Suppose $G$ is a finite group.  Then $m_G$ is increasing on $\mathcal{S}(G)$ if and only if $G = P_1 \times \cdots \times P_n$, with $p_1, \dots, p_n$ distinct primes and each $P_i$ a $p_i$-group as in Proposition \ref{prop: Cheng}.
\end{corollary}

\begin{proof}
    First suppose that $G = P_1 \times \cdots \times P_n$, with $p_1, \dots, p_n$ distinct primes and each $P_i$ a $p_i$-group as in Proposition \ref{prop: Cheng}.  Since each $P_i$ has that $\CD(P_i) = \mathcal{C}(P_i) = [P_i/\mathbf{Z}(P_i)]$, and since $(|P_i|,|P_j|)=1$ for $i \neq j$, we have that $\CD(G) = \mathcal{C}(G) = [G/\mathbf{Z}(G)]$, and so $m_G$ is increasing on $\mathcal{S}(G)$ by Theorem \ref{prev: increasing}.

    Conversely, suppose that $m_G$ is increasing on $\mathcal{S}(G)$.  Then $m_G$ is increasing on $\mathcal{C}(G)$, and so by Corollary \ref{cor: inc_decom}, we have that $G = P_1 \times \cdots \times P_n$, with $p_1, \dots, p_n$ distinct primes and each $P_i$ a $p_i$-group with $\mathcal{C}(P_i) = \CD(P_i)$.  Since $m_G$ is increasing on $\mathcal{S}(G)$, by Theorem \ref{prev: increasing}, $\CD(G) = \mathcal{C}(G) = [G/\mathbf{Z}(G)]$, and so it follows that $\CD(P_i) = \mathcal{C}(P_i) = [P_i/\mathbf{Z}(P_i)]$ for each $i$, and so each $P_i$ is as in Proposition \ref{prop: Cheng}.
\end{proof}

It seems difficult to classify the $p$-groups $G$ with $\mathcal{C}(G) = \CD(G)$. However, we can still observe many interesting properties of these groups. 

We will say that an abelian subgroup $A$ of a finite group $G$ is self-centralizing, or maximal abelian, if $A = \mathbf{C}_G(A)$.  Note that self-centralizing abelian subgroups always exist, for if $A$ is an abelian subgroup and $x \in \mathbf{C}_G(A)$, then $\langle A, x \rangle$ is abelian.

\begin{proposition}\label{prop: quasi}
Suppose $G$ is a $p$-group with $\mathcal{C}(G) = \CD(G)$, and suppose $|G| = p^a$ and $|\mathbf{Z}(G)| = p^b$.  Then $a+b$ is even and every self-centralizing abelian subgroup $A$ has order $p^{\frac{a+b}{2}}$.
\end{proposition}

\begin{proof}
$G$, $\mathbf{Z}(G)$, and any self-centralizing abelian subgroup $A$ are all in $\mathcal{C}(G)$.  And because $\mathcal{C}(G) = \CD(G)$, such subgroups all have the same Chermak--Delgado measure, and the result follows.
\end{proof}

The rest of the section investigates when $\mathcal{C}(G)$ forms a quasi-antichain, i.e., the lattice consists of a maximum, a minimum, and the atoms of the lattice.  A nonabelian group $G$ with $\mathcal{C}(G)$ a quasi-antichain is referred to as an $\mathcal{M}$-group in \cite{Schmidt_70}.

 We say that a centralizer $C$ of $G$ is minimal (maximal) if $C$ is an atom (coatom) in $\mathcal{C}(G)$.  

\begin{proposition}\label{prop: max_cent_elt}
    Let $G$ be a nonabelian finite group.  If $M$ is a maximal centralizer, then $M = \mathbf{C}_G(x)$ for some $x \in G - \mathbf{Z}(G)$.
\end{proposition}

\begin{proof}
  So $M = \mathbf{C}_G(H)$, and since $M  < G$, $H \nsubseteq \mathbf{Z}(G)$.  Let $x \in H - \mathbf{Z}(G)$.  Then $M = \mathbf{C}_G(H) \leq \mathbf{C}_G(x) < G$.  And since $M$ is a maximal centralizer, we have $M = \mathbf{C}_G(x)$.
\end{proof}

Note that the converse in not true, and so the centralizer of a non-central element does not need to be a maximal centralizer.  Take $G= D_8 \times D_8$.  Then $\mathbf{C}_G((s,s)) = \langle s,r^2 \rangle \times \langle s,r^2 \rangle < \langle s,r^2 \rangle \times D_8 = \mathbf{C}_G((s,1))$.

The next result appears in \cite{Schmidt_70}, and we suggest the reader does not look it up, but proves it as an exercise.

\begin{proposition}\label{prop: max_cent_ab}
    A finite group $G$ has that $\mathcal{C}(G)$ is a quasi-antichain if and only if every maximal centralizer of $G$ is abelian.
\end{proposition}

\begin{proposition}\label{prop: max_union}
    A nonabelian finite group $G$ is equal to the union of its maximal centralizers. 
\end{proposition}

\begin{proof}
A centralizer must contain $\mathbf{Z}(G)$.  If $x \in G-\mathbf{Z}(G)$, then $x \in \mathbf{C}_G(x) \leq M$ for some maximal centralizer $M$.  Thus $G$ is equal to the union of its maximal centralizers. 
\end{proof}

It is a well-known result that a group cannot equal the union of two of its proper subgroups.  This yields the following.

\begin{corollary}\label{cor: w_gr_2}
    In a nonabelian finite group $G$, the number of maximal (minimal) centralizers is greater than $2$. 
 In particular, if $\mathcal{C}(G)$ is a quasi-antichain, then the width of $\mathcal{C}(G)$ is greater than $2$.
\end{corollary}

\begin{proof}
That the number of maximal centralizers is greater than $2$ is a consequence of Proposition \ref{prop: max_union} and the fact that a group cannot equal the union of two of its proper subgroups.

Now the centralizer map induces a duality on $\mathcal{C}(G)$, and so the number of coatoms is equal to the number of atoms. Therefore the number of minimal centralizers is greater than $2$.
\end{proof}

Quasi-antichain Chermak--Delgado lattices were studied in \cite{Bre_14_2} and \cite{An_19}.  We present one of the main results of \cite{Bre_14_2}.

\begin{proposition}\label{prop: CD_quasi}
    If $G$ is a finite group with $\CD(G)$ a quasi-antichain of width $w \geq 3$
and $G \in \CD(G)$, then $G$ is nilpotent of class $2$; in fact, there exists a prime
$p$, a nonabelian Sylow $p$-subgroup $P$ with nilpotence class $2$, and an abelian
Hall $p'$-subgroup $Q$ such that $G = P \times Q$, $P \in \CD(P)$, and $\CD(G) \cong \CD(P)$ as lattices. Moreover, there exist positive integers $a$, $b$ with $b \leq a$ such that
$|G/\mathbf{Z}(G)| = |P/\mathbf{Z}(P)| = p^{2a}$ and $w = p^b + 1$.
\end{proposition}

Using the above results, we obtain the following.

\begin{theorem}
Suppose that $G$ is a nonabelian $p$-group.  Then $\mathcal{C}(G) = \CD(G)$ is a quasi-antichain if and only if the following all hold:
\begin{itemize}
    \item $G$ is nilpotent of class 2;
    \item $\mathcal{C}(G) = \{ G, A_1, \dots, A_w, \mathbf{Z}(G) \}$, where each $A_i$ is a maximal abelian subgroup of $G$ of order $|A_i| = \sqrt{|G||\mathbf{Z}(G)|}$; 
    \item there exist positive integers $a$, $b$ with $b \leq a$ such that
$|G/\mathbf{Z}(G)| = p^{2a}$ and $w=p^b+1$.
\end{itemize} 
\end{theorem}

\begin{proof}
   Suppose first that $\mathcal{C}(G) = \CD(G)$ is a quasi-antichain. So $G, \mathbf{Z}(G) \in \mathcal{C}(G)$, and $\mathcal{C}(G) = \CD(G) = \{ G, A_1, \dots, A_w, \mathbf{Z}(G) \}$.  The fact that $A_1,\dots,A_w$ are maximal abelian is Proposition \ref{prop: max_cent_ab}.  The fact that $w \geq 3$ is Corollary \ref{cor: w_gr_2}.  The rest follows directly from Proposition \ref{prop: quasi} and Proposition \ref{prop: CD_quasi}.

   Conversely, if $\mathcal{C}(G)$ has the prescribed properties, then since $\CD(G) \subseteq \mathcal{C}(G)$ and since $m_G$ is constant on $\mathcal{C}(G)$, we have $\mathcal{C}(G) = \CD(G)$.
\end{proof}

Let $p$ be a prime and $n$ a positive integer. Let $G$ be the group
of all $3 \times 3$ lower triangular matrices over $GF(p^n)$ with $1$'s along the diagonal.  This example was given in \cite{Bre_14_2} to show that quasi-antichain Chermak--Delgado lattices with all atoms abelian exist for all possible widths $p^n + 1$.  It is not hard to see that for each $x \in G - \mathbf{Z}(G)$, $\mathbf{C}_G(x)$ is abelian. So it follows from Proposition \ref{prop: max_cent_elt} that every maximal centralizer is abelian. By Proposition \ref{prop: max_cent_ab} we have that $\mathcal{C}(G)$ is a quasi-antichain. 
 Now there are exactly $p^n + 1$ of these maximal abelian subgroups, each of them having order $p^{2n}$, and so width of $\mathcal{C}(G)$ is $p^n + 1$, and since $m_G$ is constant on $\mathcal{C}(G)$, we have that $\CD(G) = \mathcal{C}(G)$.

\bibliographystyle{plainnat}
\bibliography{ref}

\end{document}